\documentclass[12pt]{amsart}
\usepackage{amsmath, amssymb, amsfonts, xspace, verbatim, enumitem}
\usepackage[numbers]{natbib}

\theoremstyle{plain}
\newtheorem{theorem}{Theorem}[section]
\newtheorem{proposition}[theorem]{Proposition}
\newtheorem{corollary}[theorem]{Corollary}
\newtheorem{lemma}[theorem]{Lemma}

\numberwithin{equation}{section}

\theoremstyle{definition}
\newtheorem{definition}[theorem]{Definition}

\theoremstyle{remark}
\newtheorem{remark}{Remark}

\DeclareMathOperator{\supp}{supp}
\DeclareMathOperator{\Reg}{Reg}
\DeclareMathOperator{\Lip}{Lip}
\DeclareMathOperator{\Id}{I} 

\def \< {\langle}
\def \> {\rangle}

\def \N {\mathbb{N}}
\def \R {\mathbb{R}}

\def \E {\mathbb{E}}
\def \P {\mathbb{P}}
\def \one {{\bf 1}}
\def \AA {\mathcal{A}}
\def \EE {\mathcal{E}}
\def \NN {\mathcal{N}}

\def \e {\varepsilon}
\def \d {\delta}
\def \l {\lambda}
\def \s {\sigma}
\def \S {\Sigma}
\def \hS {\hat \Sigma}
\def \const {{\mathrm const}}

\begin{document}

\title{Partial estimation of covariance matrices}
  
\author{Elizaveta Levina}
\address{Department of Statistics, 
University of Michigan, 1085 S. University, Ann Arbor, MI 48109, U.S.A.}
\email{elevina@umich.edu}

\author{Roman Vershynin} 
\address{Department of Mathematics, 
University of Michigan, 530 Church St., Ann Arbor, MI 48109, U.S.A.}
\email{romanv@umich.edu}

\date{August 9, 2010, revised February 11, 2011}

\thanks{Partially supported by NSF grants DMS 0805798 (E.L.) and FRG DMS 0918623, DMS 1001829 (R.V.)}

\begin{abstract}
A classical approach to accurately estimating the covariance matrix 
$\Sigma$ of a $p$-variate normal distribution is to draw a sample of size $n > p$ and form a sample covariance matrix. However, many modern applications 
operate with much smaller sample sizes, thus calling for estimation guarantees 
in the regime $n \ll p$.  We show that a sample of size $n = O(m \log^6 p)$ is sufficient to accurately estimate in operator norm an arbitrary symmetric part of $\Sigma$ consisting of $m \le n$ nonzero entries per row.  This follows from a general result on estimating Hadamard products $M \cdot \Sigma$, where $M$ is 
an arbitrary symmetric matrix.   
\end{abstract}

\maketitle

\section{Introduction}

The problem of estimating the population covariance matrix from a sample of $n$ i.i.d.\ observations 
$X_1, \dots, X_n$  in $\R^p$ has attracted a lot of recent interest in statistics, 
with the focus on small sample sizes, $n \ll p$.  
Estimation of covariance matrices plays a key role in many data analysis techniques (e.g. in principal component 
analysis, discriminant analysis, graphical models). 
Applications where $n \ll p$ have become very common in gene expression data, 
climate studies, spectroscopy, and so on. 

Suppose $X_k$'s are drawn from the multivariate normal distribution $\NN_p(0, \S)$, 
where $\S =  \E X X^T$ is a symmetric positive-semidefinite $p \times p$ matrix. Thus
$\S$ is the covariance matrix of that distribution. The usual estimator for $\S$ is 
the sample covariance matrix defined by\footnote{Centering the data 
  with zero mean rather than the sample mean 
  is essentially without loss of generality, since the mean term is of higher 
  order (see Remark~\ref{remark centering} after Corollary~\ref{cor: sparse estimation}).}
\begin{equation}				\label{def:sample}
\hat \S_n = \frac{1}{n} \sum_{k=1}^n X_k X_k^T \ .
\end{equation}
The properties of $\hS_n$ have been extensively studied
\citep{marcenko67, johnstone01, johnstone04}. 
Let us first consider the simpler case where $\S=\Id$. 
The ``Bai-Yin law'' of random matrix theory determines 
the magnitude of the extreme eigenvalues of the Wishart random matrix $\hat \S_n$ in the limit
as $n, p \to \infty$ and $n/p \to \const$.
It says that the spectrum of $\hat \S_n$ is almost surely
contained in the interval $[\frac{a^2}{n} + o(1), \frac{b^2}{n} + o(1)]$
where $a = (\sqrt{n}-\sqrt{p})_+$ and $b = \sqrt{n} + \sqrt{p}$. 
(The upper bound is due to Geman~\cite{Geman}, the lower bound is due to Silverstein~\cite{Silverstein}; 
see \cite{Bai-Yin} for a unified exposition). 
It follows that, with high probability, 
$$
\|\hS_n - \Id\| \leq 2 \sqrt{\frac{p}{n}} + \frac{p}{n} + o(1)
$$
where $\| \cdot \|$ is the spectral norm, also known as the operator or $\ell_2$ matrix norm.  

A similar approximation holds for random vectors $X_k$ drawn from 
a general multivariate normal distribution $\NN_p(0, \S)$ (where $\S$ depends on $p$).
Indeed, arguing as above for the random vectors $\S^{-1/2} X_k$ whose covariance
matrix is identity, we easily conclude that with high probability,
\begin{equation}		\label{eq:sample}
\|\hS_n - \S\| \leq \Big( 2 \sqrt{\frac{p}{n}} + \frac{p}{n}  + o(1) \Big) \, \|\S\|.
\end{equation}
This implies that for an arbitrary precision $\e \in (0,1)$ and for $p$ sufficiently large,
the sample size 
\begin{equation}				\label{e-close} 
n \gtrsim \e^{-2} p \quad \text{suffices for} \quad \|\hS_n - \S\| \le \e \|\S\|.
\end{equation} 
More accurately, the following non-asymptotic result holds for arbitrary $p \in \N$,
$\e \in (0,1)$ and $t \ge 1$. If the sample size satisfies $n \ge C(t/\e)^2 p$, then inequality \eqref{e-close} 
holds with probability at least $1 - 2\exp(-t^2 n)$ \cite[Remark 51]{v-rmt-tutorial}. 
Here $C$ is an absolute constant.
This provides a satisfactory answer to the covariance estimation problem
in the regime $n \ge p$ for general $p$-variate normal distributions.

We will now focus on the more difficult regime $n < p$, where covariance estimation 
is impossible in general (indeed, 
if $\Sigma=I$ then clearly $\|\hS_n - I\| \geq 1$ because of the rank deficiency of $\hS_n$). 
Still, under natural structural assumptions on the covariance matrix $\S$,
a number of alternative estimation schemes have been proposed. 
  
A common assumption is that most entries of $\S$ are zero or close to 0, 
and thus can be safely estimated as 0, which reduces the effective size of the problem.  
This assumption underlies banding or tapering the covariance matrix in the case of ordered variables 
\cite{bl06,furrer06,rothman10,cai08}, and thresholding of the covariance matrix 
in the case of unordered variables \cite{bl07, elkaroui07, genthresh, cai10}. 
In both cases, the regularized estimators have the form $M \cdot \hS_n$ 
where $M$ is a regularizing ``mask'' matrix, and ``$\cdot$'' denotes the Hadamard (entrywise) product of
matrices. For banding and hard thresholding, $M$ is a sparse matrix containing only 0s and 1s;  
for tapering, its entries can be anywhere in the interval $[0,1]$, depending on the taper used.   
For thresholding, the mask $M$ has to be estimated from $\hS$, 
whereas for banding and tapering it is determined by a tuning parameter.

The reason regularized estimators work is best understood from the following decomposition:  
$$
\| M \cdot \hS_n - \S \| \le \| M \cdot \hS_n -  M \cdot \S \| + 
\| M \cdot \S - \S \| \ . 
$$
The first term (variance) is well-behaved because $M$ is sparse or close to sparse, and
the second term (bias) is well-behaved because of the assumptions made on the true $\Sigma$.  
The overall rate for covariance estimation is normally computed as the sum of these two terms, 
but they come from two quite different problems:  the first term has to do with how much 
of the covariance matrix we can reliably estimate with the corresponding part of the sample covariance matrix;  
and the second term has to do with our model assumptions.  

In this paper, we focus on the first question: {\em  how large does the 
sample size $n$ have to be so that $M \cdot \hS_n$ is a good approximation 
to $M \cdot \S$}?  Of course, the answer 
will depend on the size of the part we would like to estimate,
but it almost does not depend on the ambient dimension $p$ as we will show.

As a simple example, consider the case where $M = (m_{ij})$ is a fixed minor, 
i.e. $m_{ij} = \one_{\{i,j \in S\}}$
for some a priori given index set $S \subset \{1, \dots, p\}$. Denoting $|S|=m$, 
we can easily deduce from the results above that 
\begin{equation}								\label{minor}
\|M \cdot \hS_n - M \cdot \S\| 
\leq \Big( 2 \sqrt{\frac{m}{n}} + \frac{m}{n} + o(1) \Big) \, \|\S\|.
\end{equation}
Note that the quality of approximation does not depend on the ambient dimension $p$
but only on the number of selected variables $m$, as one would expect. 

\medskip

The purpose of this paper is to reveal a general phenomenon behind \eqref{minor}.
We shall show that, up to a necessary logarithmic factor in $p$, 
the upper bound in \eqref{minor} {\em holds for arbitrary symmetric $0/1$ matrices $M$ with 
at most $m$ nonzero entries per column}. This result is independent of any structural hypotheses
on the locations of the nonzero entries. 

Even more generally, a stable version of \eqref{minor} holds. 
It applies to completely {\em arbitrary symmetric matrices $M$},
and therefore covers all forms of tapering previously considered in the literature.    
Instead of the sparsity parameter $m$, the bound is governed by two different operator norms of $M$
that correspond to the different dependencies on $m$ in the two terms in \eqref{minor}.   
These norms are the $\ell_1 \to \ell_2$ operator norm $\|M\|_{1,2} = \max_j (\sum_i m_{ij}^2)^{1/2}$, 
and the usual $\ell_2 \to \ell_2$ operator norm $\|M\|$. 

\medskip

Note that since the right hand side of \eqref{e-close} and \eqref{minor} depends on $\| \Sigma\|$, these are relative rather than the absolute error bounds on the difference between $\hS$ and $\S$.  The latter are more commonly analyzed in the literature, even though numerical results are commonly reported as relative errors to enable comparisons across models.  The advantage of using the relative error is disentangling the convergence rates of the estimator 
from assumptions on the true covariance $\Sigma$, which is one of our goals in this paper.

\section{Main result}

Throughout this paper, we consider an arbitrary mean 0 normal distribution in $\R^p$.
Its covariance matrix is denoted by $\Sigma$, and the sample covariance matrix \eqref{def:sample} obtained 
from an i.i.d.\ sample of $n$ observations by $\hS_n$.
Positive absolute constants will be denoted $C, C_1, C_2$.
\begin{theorem}	[Estimation of Hadamard products]				\label{main}   
  Let $M$ be an arbitrary fixed symmetric $p \times p$ matrix. Then  
  \begin{equation}										\label{eq: main}
  \E \|M \cdot \hS_n - M \cdot \S\| 
  \le C \log^3(2p) \Big( \frac{\|M\|_{1,2}}{\sqrt{n}} + \frac{\|M\|}{n} \Big) \, \|\S\|.
  \end{equation}
\end{theorem}
Note that $M$ does not depend on $\hS_n$ or $\S$.  We will prove this result in Section \ref{s: proof}.

\begin{corollary}[Partial estimation]					\label{cor: sparse estimation}		
\label{partial}
  Let $M$ be an arbitrary fixed symmetric $p \times p$ matrix such that all of its entries 
  are equal to $0$ or $1$, and there are at most $m$ nonzero entries in each column. 
  Then 
  $$
  \E \|M \cdot \hS_n - M \cdot \S\| 
  \le C \log^3(2p) \Big(\sqrt{\frac{m}{n}} + \frac{m}{n} \Big) \, \|\S\|.
  $$ 
\end{corollary}

\begin{proof}
We note that $\|M\|_{1,2} \le \sqrt{m}$ and $\|M\| \le m$ and apply Theorem~\ref{main}.
\end{proof}

\begin{remark}[Sample size]
Corollary~\ref{cor: sparse estimation} implies that for every $\e \in (0,1)$, the sample size
\begin{equation}								\label{nmp}
n \ge 4C^2 \e^{-2} m \log^6(2p)
\quad \text{suffices for} \quad
\E \|M \cdot \hS_n - M \cdot \S\| \le \e \|\S\|.
\end{equation}
For sparse matrices $M$ with $m \ll p$, this makes partial estimation possible with $n \ll p$ observations. 
Therefore we regard \eqref{nmp} as a satisfactory ``sparse'' version of the classical bound \eqref{e-close}. 
A logarithmic term necessarily has to appear in \eqref{nmp}, see Remark~\ref{r: log} below.
\end{remark}

\begin{remark}[Compressed sensing]
  Results of a similar nature arise in {\em compressed sensing} (see e.g. \cite{fornasier-rauhut-cs}). 
  There one tries
  to reconstruct a signal $x \in \R^p$ from its $n$ linear measurements $y = Ax \in \R^n$.
  In the range $n \ge p$, where the problem is overdetermined, the reconstruction is achieved by inverting the $n \times p$ matrix $A$
  (assuming it has full rank). In the range $n \ll p$, the problem is underdetermined, and a reconstruction is only possible 
  under certain structural assumptions on the signal $x$ such as sparsity. If $x$ has $m$ non-zero 
  coordinates, then the typical results of compressed sensing guarantee an exact and algorithmically 
  effective reconstruction provided that $n \gtrsim m \log^C p$. 

  Our result \eqref{nmp} is of a similar nature:
  $n \sim m \log^6 p$ observations suffice to estimate an $m$-sparse part of a covariance matrix in dimension $p$.
  A significant difference is that the compressed sensing problem becomes easy if 
  one knows the locations of the non-zero coordinates of the signal $x$. 
  In contrast, the covariance estimation problem remains non-trivial even if one knows 
  the location of the $m$-sparse part of the covariance matrix 
  (which is the situation in Corollary~\ref{cor: sparse estimation}).
\end{remark}

\begin{remark}[Logarithmic factor]					\label{r: log}
  In Corollary~\ref{partial} note the mild, logarithmic dependence on 
  the dimension $p$ and the optimal dependence on the sparsity $m$, 
  which is of the same type as in \eqref{minor}.   The necessity of the logarithmic term in our 
  results is evident from the example where $\S = M = \Id$. 
  Indeed, writing in coordinates $X_k = (X_{k1}, \ldots, X_{kp}) \in \R^p$, we obtain
  $$
  \E \|M \cdot \hS_n - M \cdot \S\| 
  = \E \max_{i \in [p]} \Big| \frac{1}{n} \sum_{k=1}^n X_{ki}^2 - 1 \Big| 
  \sim  \sqrt{\frac{\log(2p)}{n}}.
  $$
  We may compare this with the conclusion of Theorem~\ref{main} and Corollary~\ref{partial} for this 
  example, which is
  $$
  \E \|M \cdot \hS_n - M \cdot \S\| 
  \le C \frac{\log^3(2p)}{\sqrt{n}}.
  $$
  We see that the logarithmic terms in these results are unavoidable.
  However, the present exponent $3$ of the logarithm is certainly not optimal, 
  and it can probably be reduced to the optimal value $1/2$.
\end{remark}

\begin{remark}[Centering]		\label{remark centering}
If we center the data with the sample mean 
$\bar X = \frac{1}{n} \sum_{k = 1}^n X_k$ rather than the true mean, as one would in practice, we have 
$$
\E \| M \cdot (\S_n - \bar X \bar X^T) - M \cdot \S \| \le 
\E \| M \cdot \S_n - M \cdot \S \| + \E \| M \cdot (\bar X \bar X^T) \|
$$ 
Then, denoting by $\|x\|_\infty = \max_{i} |x_i|$ the $\ell_\infty$ norm in $\R^p$, we can check that
$$
\E \| M \cdot (\bar X \bar X^T) \| 
\le \E \|M\| \|\bar X\|_\infty^2 
\le C \|M\| \|\Sigma\| \frac{\log(2p)}{n},
$$
so this term can be absorbed into the second term in \eqref{eq: main}.
\end{remark}

\begin{remark}[Identifying the non-zero entries of $\Sigma$ by thresholding]
  For thresholding covariance matrices, our result is not directly applicable 
  since the matrix $M$ is random and estimated from data.   
  However, we note that if the assumption is made that all non-zero entries in $\Sigma$ 
  are bounded away from zero by a margin of $h > 0$, then a sample size of $n \gtrsim h^{-2} \log(2p)$ 
  would assure that all their locations are estimated correctly with probability approaching 1.   
  With this assumption, we could derive a bound for the thresholded estimator.  
  Since our focus here is not on any one particular estimator, we do not pursue this extension, 
  particularly since there are already bounds available specifically for thresholding \cite{bl07,cai10}. 
\end{remark}

\begin{remark}[Previous results]
  Results similar to Corollary~\ref{cor: sparse estimation} have been known for {\em banding} covariance
  matrices. They apply to the mask matrices $M$ with $m$ diagonals consisting of $1$'s (while all other entries are zero). 
  A result of \cite{bl06} guarantees in this case an error bound of order $m \sqrt{\log p/n}$. 
  While this bound has a slightly better dependence on $p$, 
  it implies that the sample size $n$ required for estimation grows as $m^2$, 
  whereas our result shows $n$ growing linearly with $m$.  
  Furthermore, for some tapered version of the $m$-diagonal matrix $M$, an error of order
  $\sqrt{(m + \log p)/n}$ was obtained in \cite{cai08}, which agrees with ours up to $\log p$. 
  The latter result heavily exploits the structure of the tapered banded matrix $M$, which allows one 
  to decompose it into a weighted sum of small minors and reduce the problem to using 
  \eqref{minor} for each minor. In contrast, our result does not require any structure of the mask matrix $M$. 
\end{remark}

\subsection*{Acknowledgement}
  The authors are grateful to the anonymous referee for careful reading of the manuscript and useful suggestions.

\section{Preliminaries}							\label{s: preliminaries}

\subsection{Proof outline} 
In the rest of this paper we prove Theorem~\ref{main}. 
We shall observe that the quadratic form $\< (M \cdot \hS_n)x, y \> $
is a Gaussian chaos (defined in \eqref{gchaos} below) 
for fixed unit vectors $x,y$ on the sphere $S^{p-1}$.
The main difficulty is to control the chaos uniformly for all $x,y$. 
This will be done by establishing concentration inequalities 
of varying power depending on the ``sparsity'' of $x,y$ 
(which amounts to decoupling, conditioning, and Gaussian concentration), 
and combining this with covering arguments to ``count'' the number of 
sparse vectors $x,y$ on the sphere.

\subsection{Decoupling}

The following observation does not seem to be found in 
literature on decoupling (such as \cite{delapena-gine-decoupling}). 

\begin{proposition}[Decoupling of sample covariance matrices]				\label{decoupling}
  Let $X$ be a centered random normal vector in $\R^p$ with covariance matrix $\S$, and let 
  $X_1,\ldots, X_n$ and $X'_1, \ldots, X'_n$ be independent copies of $X$.
  Consider the sample covariance matrix and the decoupled sample covariance matrix of $X$ defined as
  \begin{equation}							\label{Sn'}
  \hS_n = \frac{1}{n} \sum_{k=1}^n X_k \otimes X_k, \quad 
  \S'_n = \frac{1}{n} \sum_{k=1}^n X'_k \otimes X_k.
  \end{equation}
  Then for every symmetric $p \times p$ matrix $M$ we have
  $$
  \E \| M \cdot \hS_n - M \cdot \S\| \le 2 \E \|M \cdot \S'_n\|.
  $$ 
\end{proposition}

\begin{proof}
Computing the operator norm as the maximal value of the associated
quadratic form, we have 
\begin{align*}
\E \| M \cdot \hS_n - M \cdot \S\|
&= \E \| M \cdot \hS_n - \E (M \cdot \hS_n)\| \\
&= \E \sup_{a \in S^{p-1}} |\< (M \cdot \hS_n) a, a \> - \E \< (M \cdot \hS_n) a, a \> |.
\end{align*}
Writing the inner product in coordinates
$$
\< (M \cdot \hS_n) a, a \> = \frac{1}{n} \sum_{k=1}^n \sum_{i,j=1}^p m_{ij} a_i a_j X_{k i} X_{k j}
$$
we recognize in the right hand side a {\em quadratic Gaussian chaos}. 
Recall that generally, a quadratic Gaussian chaos in dimension $p$ is a quadratic form 
\begin{equation}
\sum_{i,j=1}^p a_{ij} Z_i Z_j  = \< AZ,Z \>
\label{gchaos}
\end{equation}
where $Z=(Z_1,\ldots,X_p)$ is a centered normal random vector in $\R^p$, 
and the coefficients $A = (a_{ij})_{i,j=1}^p$
are taken from some subset $\AA$ of $p \times p$ matrices. 
The conclusion of the proposition then follows from the next lemma. 
\end{proof}

\begin{lemma}[Decoupling of a Gaussian chaos]				\label{decoupling chaos}
  Let $Z$ be a centered normal random vector in $\R^d$, and let $Z'$ be an independent copy of $Z$. 
  Let $\AA$ be a subset of symmetric $d \times d$ matrices. Then 
  $$
  \E \sup_{A \in \AA} | \< AZ,Z\> - \E \< AZ,Z\> |
  \le 2 \E \sup_{A \in \AA} | \< AZ,Z'\> |.
  $$
\end{lemma}

Lemma~\ref{decoupling chaos} allows one to replace a Gaussian chaos $\sum_{i,j} a_{ij} Z_i Z_j$
by its decoupled version $\sum_{i,j} a_{ij} Z'_i Z_j$.
While it can be derived from observations in \cite[Section 3.2]{LT},
it is simpler to give a complete proof.

\begin{proof}
Without loss of generality, we may assume that $Z$ is a standard normal random vector. 
Indeed, there exists a symmetric $d \times d$ matrix $T$ such that $Z = Tg$ and $Z' = Tg'$
for some independent standard Gaussian vectors $g, g'$ in $\R^d$.
Thus $\< AZ,Z\> = \< TATg,g\> $ and $\< AZ,Z'\> = \< TATg,g'\> $. Replacing
$A$ by $TAT$ in the statement of the lemma, we see that we can assume that 
$Z$ is a standard normal random vector. 

Using the identical distribution of $Z$ and $Z'$ and Jensen's inequality, we obtain
\begin{align*}
E &:= \E \sup_{A \in \AA} | \< AZ,Z\> - \E \< AZ,Z\> |
= \E \sup_{A \in \AA} | \< AZ,Z\> - \E \< AZ',Z'\> | \\
&\le \E \sup_{A \in \AA} | \< AZ,Z\> - \< AZ',Z'\> |.
\end{align*}
Note the identity 
$\< AZ,Z\> - \< AZ',Z'\> = 2 \big\langle A ( \frac{Z+Z'}{\sqrt{2}} ), \frac{Z-Z'}{\sqrt{2}} \big\rangle$.
By rotation invariance of Gaussian measure, the pair of vectors
$\big( \frac{Z+Z'}{\sqrt{2}}, \frac{Z-Z'}{\sqrt{2}} \big)$ is distributed identically 
with $(Z,Z')$. Hence we conclude that 
$$
E \le 2 \E \sup_{A \in \AA} | \< AZ,Z'\> |.
$$
This completes the proof.
\end{proof}

\begin{remark}
It is not clear if a version of Decoupling Proposition~\ref{decoupling} holds for general distributions. 
The argument we gave relied heavily on the rotation invariance of the normal distribution.
\end{remark}

\subsection{Concentration}

The following observation is a form of the rotation invariance property of
normal distributions. 

\begin{lemma}								\label{sum of normals}
  Let $Z = (Z_1,\ldots,Z_d)$ be a centered normal random vector in $\R^d$ with 
  covariance matrix $\S$. Let $a = (a_1,\ldots,a_d) \in \R^d$.
  Then $\sum_{i=1}^d a_i Z_i$ is a centered normal random variable 
  with standard deviation 
  $\|\S^{1/2} a\|_2 \le \|\S\|^{1/2} \|a\|_2$. \qed
\end{lemma}

Our proof of Theorem~\ref{main} will use the concentration of measure in the Gauss space. 
Such concentration results are usually stated in the literature for the
standard normal distribution, but it is straightforward to make an adjustment for 
general normal distributions.

\begin{proposition}[Concentration in the Gauss space, see \cite{LT} Section 1.1]				\label{concentration}
  Let $f: \R^d \to \R$ be a Lipschitz function, and denote $L = \|f\|_{\Lip}$. 
  Let $Z$ be a centered normal vector in $\R^d$ with covariance matrix $\S$.
  Then 
  $$
  \P \{ f(Z) - \E f(Z) \ge t \} \le \frac{1}{2} \exp \Big( -\frac{t^2}{2L^2\|\S\|} \Big) \quad \text{for all } t \ge 0.
  $$
\end{proposition}

\begin{remark}
  If $f$ is a norm on $\R^d$, then by translation invariance,
  $\|f\|_{\Lip}$ equals the minimal number $L$ such that 
  $$
  f(x) \le \|x\|_2 \quad \text{for all } x \in \R^d.
  $$
\end{remark}

\subsection{Discretization}

The operator norm of a $p \times p$ matrix $A$ can be computed via
the associated bilinear form as
$$
\|A\| = \max_{x,y \in S^{p-1}} \< Ax,y\> 
$$
By an approximation argument, one can replace the sphere $S^{p-1}$ by 
an $\e$-net:

\begin{lemma}[Computing operator norm on a net, see \cite{v-rmt-tutorial}]			\label{on net}
  Let $A$ be a $p \times p$ matrix. Let $\NN$ be a $\d$-net of $S^{p-1}$ in the 
  Euclidean metric for some $\d \in [0,1)$. Then 
  $$
  \|A\| \le (1-\d)^{-2} \max_{x,y \in \NN} \< Ax,y\> .
  $$
\end{lemma}

We will construct a net $\NN$ by rounding off the coefficients of a vector $x$.

\begin{definition}[Regular vectors]
  Define the subset of {\em regular vectors} of the sphere $S^{p-1}$ as follows:
  \begin{align*}
  &\Reg_p(s) = \big\{ x \in S^{p-1}: \text{all coordinates satisfy } x_i^2 \in \{0, 1/s\} \big\}, \quad s \in [p]; \\
  &\Reg_p = \bigcup_{s \in [p]} \Reg_p(s).
  \end{align*}
  Note that $|\supp(x)| = s$ for all $x \in \Reg_p(s)$.
\end{definition}

\begin{lemma}[Computing operator norm on regular vectors]			\label{on Reg}
  Let $A$ be a $p \times p$ matrix. Then
  $$
  \|A\| \le 12 \lceil \ln(2p) \rceil^2 \max_{x,y \in \Reg_p} \< Ax, y\> . 
  $$
\end{lemma}

\begin{proof}
Define $k_0 = \lceil \ln(2p) \rceil$ and let $\NN$ consist of all vectors $x \in S^{p-1}$ such that for each coordinate
$x_i$ there is $k \in \{0,1,\ldots,k_0\}$ such that $|x_i| = 2^{-k}$.
It is easy to check that $\NN$ is a $0.71$-net of $S^{p-1}$.

By considering the level sets of a vector $x \in \NN$, one can represent 
$x = \l_1 x^{(1)} + \cdots +  \l_{k_0} x^{(k_0)}$ where
all $\l_k \in [0,1]$ and $x^{(k)}$ are regular vectors. 
Using this representation for $x$ and $y$ in Lemma~\ref{on net} with $\d = 0.71$,
we conclude by the linearity of the inner product that the estimate in the lemma holds. 
\end{proof}

\begin{lemma}[Computing operator norm: symmetric distributions]			\label{on symmetric}
  Let $A$ be a random $p \times p$ matrix such that $A$ and $A^T$ are identically 
  distributed. Then, for every $t \ge 0$, we have
  $$
  \P \{ \|A\| \ge t \}  
  \le 2 \P \Big\{ 12 \lceil \ln(2p) \rceil^2 \max_{\substack{r,s \in [p] \\ r \le s}} 
    \max_{\substack{x \in \Reg_p(r) \\ y \in \Reg_p(s)}} \< Ax, y\> \ge t \Big\}.
  $$  
\end{lemma}

\begin{proof}
We can write the conclusion of Lemma~\ref{on Reg} as 
$$
\|A\| 
\le 12 \lceil \ln(2p) \rceil^2 \max_{r,s \in [p]} 
  \max_{\substack{x \in \Reg_p(r) \\ y \in \Reg_p(s)}} \< Ax, y\> . 
$$  
Taking the maximum separately over $r \le s$ and $r > s$ and replacing 
the maximum by the sum, we have
\begin{equation}								\label{two max max}
\|A\| 
\le 12 \lceil \ln(2p) \rceil^2 
  \max \Big( \max_{\substack{r,s \in [p] \\ r \le s}} 
  \max_{\substack{x \in \Reg_p(r) \\ y \in \Reg_p(s)}} \< Ax, y\> , \; 
  \max_{\substack{r,s \in [p] \\ s \le r}} 
  \max_{\substack{x \in \Reg_p(r) \\ y \in \Reg_p(s)}} \< Ax, y\> \Big).
\end{equation}
By the symmetry assumption on the distribution of $A$, the random variables
$\< Ax,y \> $ and $\< Ay,x \> $ are identically distributed for all $x,y$.
Therefore, the two double maxima in \eqref{two max max} are identically distributed. 
This yields the conclusion of the lemma.
\end{proof}

\section{Proof of Theorem~\ref{main}}				\label{s: proof}

\subsection{Decoupling and conditioning}							\label{s: decoupling}

By rescaling, we can assume without loss of generality that $\|\S\| = 1$.
We shall denote the entries of $M$ by $m_{ij}$.

By Proposition~\ref{decoupling}, it suffices to estimate $\E \|M \cdot \S'_n\|$, 
where $\S'_n$ denotes the decoupled sample covariance matrix defined in \eqref{Sn'}. 
By the definition of $\S'_n$ and the symmetry of $M$, the random matrices $M \cdot \S'_n$ 
and $(M \cdot \S'_n)^T$ are identically distributed. Therefore, Lemma~\ref{on symmetric}
applies and we get for every $t \ge 0$ that 
\begin{equation}					\label{max max}
\P \{ \|M \cdot \S'_n\| \ge t \} 
\le 2 \P \Big\{ 12 \lceil \ln(2p) \rceil^2
  \max_{\substack{r,s \in [p] \\ r \le s}} 
  \max_{\substack{x \in \Reg_p(r) \\ y \in \Reg_p(s)}} 
  \< (M \cdot \S'_n)x, y\> \ge t \Big\}.
\end{equation}
Writing the inner product in coordinates and rearranging the terms, we have 
\begin{equation}							\label{triple sum}
\< (M \cdot \S'_n)x, y\> 
= \frac{1}{n} \sum_{k=1}^n \sum_{i=1}^p \Big( \sum_{j=1}^p m_{ij} x_j X_{kj} \Big) y_i X'_{ki}.
\end{equation}

Let us fix $x$ and $y$ and condition on the random vectors $(X_1,\ldots, X_n)$.
Then the expression in \eqref{triple sum} defines a centered normal random variable. We shall estimate
its standard deviation by Lemma~\ref{sum of normals}. Since the covariance matrix $\S$ of each vector
$X'_k \in \R^p$ has operator norm at most $1$, the covariance matrix of the concatenated vector 
$(X'_k)_{k=1}^n \in \R^{pn}$ also has operator norm at most $1$. 
Then Lemma~\ref{sum of normals}
yields that the expression in \eqref{triple sum} is a centered normal random variable with standard deviation
at most $\s_x(X_1,\ldots,X_n) \|y\|_\infty$ where 
$$
\s_x(X_1,\ldots,X_n) 
= \frac{1}{n} \Big[ \sum_{k=1}^n \sum_{i=1}^p  \Big( \sum_{j=1}^p m_{ij} x_j X_{kj} \Big)^2 \Big]^{1/2}.
$$
We will need to bound this quantity uniformly for all $x$.

\subsection {Concentration}

Let $x \in \Reg_p(r)$. 
We will estimate $\s_x(X_1,\ldots,X_n)$ using concentration in Gauss space, Proposition~\ref{concentration}.
First, we have
\begin{align}										\label{E sigma}
\E \s_x(X_1,\ldots,X_n) 
&\le (\E \s_x(X_1,\ldots,X_n)^2)^{1/2} \notag\\
&= \frac{1}{n} \Big[ \sum_{k=1}^n \sum_{i=1}^p  \E \Big( \sum_{j=1}^p m_{ij} x_j X_{kj} \Big)^2 \Big]^{1/2} \notag\\
&\le \frac{1}{n} \Big( \sum_{k=1}^n \sum_{i=1}^p \sum_{j=1}^p m_{ij}^2 x_j^2 \Big)^{1/2}
  \quad \text{(by Lemma~\ref{sum of normals} and $\|\S\|=1$)} \notag\\
&\le \frac{1}{\sqrt{n}} \max_{j \in [p]} \Big( \sum_{i=1}^p m_{ij}^2 \Big)^{1/2}
  \quad \text{(because $\|x\|_2 = 1$)} \notag\\
&\le \frac{\|M\|_{1,2}}{\sqrt{n}} 
  \quad \text{(by definition).}
\end{align}

Next, we consider $\s_x: \R^{pn} \to \R$
as a function of the concatenated Gaussian vector 
$(X_1,\ldots,X_n) \in \R^{pn}$.
Computing the Lipschitz norm of $\s_x$ becomes easy once we write this function as
$$
\s_x(X_1,\ldots,X_n) = \frac{1}{n} \Big( \sum_{k=1}^n \| M (x \cdot X_k)\|_2^2 \Big)^{1/2}
$$
where as usual $x \cdot X_k$ denotes the Hadamard (coordinate-wise) product of vectors,
and the multiplication by $M$ is the ordinary (matrix) multiplication.
Separating $M$ and $x$, we obtain the bound
$$
\s_x(X_1,\ldots,X_n) \le \frac{1}{n} \|M\| \|x\|_\infty \Big( \sum_{k=1}^n \|X_k\|_2^2 \Big)^{1/2}.
$$
Using that $\|x\|_\infty = 1/\sqrt{r}$ since $x \in \Reg_p(r)$, we conclude that 
$$
\s_x(X_1,\ldots,X_n) \le \frac{\|M\|}{\sqrt{r} \, n} \cdot \|(X_1,\ldots,X_n)\|_2.
$$
By the remark below Proposition~\ref{concentration}, we have proved that
\begin{equation}								\label{Lip}
\|\s_x\|_{\Lip} \le \frac{\|M\|}{\sqrt{r} \, n}.
\end{equation}
In addition, since the covariance matrix $\S$ of each vector $X_k \in \R^p$ 
has operator norm at most $1$, the covariance matrix of the concatenated vector 
$(X_1,\ldots,X_n) \in \R^{pn}$ also has operator norm at most $1$.
From this and the bounds on the expectation \eqref{E sigma} and on the Lipschitz norm \eqref{Lip},
we conclude by Proposition~\ref{concentration} that for all $x \in \Reg_p(r)$ and $t \ge 0$,
\begin{equation}								\label{sx bound}
\P \Big\{ \s_x(X_1,\ldots,X_n) > \frac{\|M\|_{1,2}}{\sqrt{n}} + t \Big\} 
\le \frac{1}{2} \exp \Big( -\frac{t^2 r n^2}{2\|M\|^2} \Big).
\end{equation}

\subsection{Union bounds}

We return to estimating the random variable $\< (M \cdot \S'_n)x, y\> $
which we initiated in Section~\ref{s: decoupling}.
Let us fix $u \ge 1$. For each $x \in \Reg_p(r)$, we consider the events 
$$
\EE_x := \Big\{ \s_x(X_1,\ldots,X_n) 
  \le \frac{\|M\|_{1,2}}{\sqrt{n}} + u \frac{\|M\|}{n} \sqrt{10\ln(2ep)} \Big\}
$$
By \eqref{sx bound}, we have
\begin{equation}								\label{EEx}
\P(\EE_x) \ge 1 - \frac{1}{2} \exp \big( -5 u^2 r \ln(2ep) \big).
\end{equation}
Note that the function $\s_x(X_1,\ldots,X_n)$ and thus also the events $\EE_x$ are independent of the random variables 
$(X'_1,\ldots, X'_n)$. 

Let $x \in \Reg_p(r)$ and $y \in \Reg_p(s)$. 
As we noted in Section~\ref{s: decoupling}, conditioned on 
a realization of random variables $(X_1,\ldots,X_n)$ satisfying $\EE_x$, 
the random variable $\< (M \cdot \S'_n)x, y\> $
is distributed identically with a centered normal random variable $h$ 
whose standard deviation is bounded by
$$
\s_x(X_1,\ldots,X_n) \|y\|_\infty 
\le \frac{\|M\|_{1,2}}{\sqrt{sn}} + u \frac{\|M\|}{\sqrt{s} \, n} \sqrt{10\ln(2ep)} 
=: \s.
$$
Then by the usual tail estimate for Gaussian random variables, we have
$$
\P \{ \< (M \cdot \S'_n)x, y\> \ge \e \;|\; \EE_x \} 
\le \frac{1}{2} \exp(-\e^2/2\s^2)
\quad \text{for } \e \ge 0.
$$
Choosing
\begin{equation}								\label{epsilon}
\e = \e(u) := 2 \sqrt{3} \, u \frac{\|M\|_{1,2}}{\sqrt{n}} \sqrt{\ln(2ep)} 
  + 2 \sqrt{30} \, u^2 \frac{\|M\|}{n} \ln(2ep),
\end{equation}
we obtain 
$$
\P \big\{ \< (M \cdot \S'_n)x, y\> \ge \e \;|\; \EE_x \big\} 
\le \frac{1}{2} \exp \big( - 6 u^2 s \ln(2ep) \big)
$$
for all $x \in \Reg_p(r)$, $y \in \Reg_p(s)$.
We would like to take the union bound in this estimate over all $y \in \Reg_p(s)$. Note that 
$$
|\Reg_p(s)| = \binom{p}{s} 2^s \le \exp \big(s \ln (2ep/s) \big)
$$
as there are $\binom{p}{s}$ ways to choose the support  
and $2^s$ ways to choose the signs of the coefficients of a vector in $\Reg_p(s)$. 
Then 
\begin{align*}
\P \big\{ \max_{y \in \Reg_p(s)} \< (M \cdot \S'_n)x, y\> \ge \e \;|\; \EE_x \big\} 
&\le  \frac{1}{2} \exp \big( s \ln (2ep/s) \big) \exp \big( - 6 u^2 s \ln(2ep) \big) \\
&\le  \frac{1}{2} \exp \big( - 5 u^2 s \ln(2ep) \big)
\end{align*}
as $u \ge 1$.
Therefore, using \eqref{EEx}, we have
\begin{align}				\label{exp+exp}
\P \big\{ \max_{y \in \Reg_p(s)} &\< (M \cdot \S'_n)x, y\> \ge \e \big\} 
\le \P \big\{ \max_{y \in \Reg_p(s)} \< (M \cdot \S'_n)x, y\> \ge \e \;|\; \EE_x \big\} + \P\{ \EE_x^c\} \notag\\
&\le \frac{1}{2} \exp \big( - 5 u^2 s \ln(2ep) \big) + \frac{1}{2} \exp \big( - 5 u^2 r \ln(2ep) \big).
\end{align}
Now we take a further union bound over $x \in \Reg_p(r)$ for fixed $r,p \in [p]$, $r \le s$.
In this range, the second term in \eqref{exp+exp} dominates. Estimating as before 
$|\Reg_p(r)| \le \exp(r \ln (2ep/r))$, we obtain that
\begin{align*}
\P \big\{ \max_{\substack{x \in \Reg_p(r) \\ y \in \Reg_p(s)}} \< (M \cdot \S'_n)x, y\> \ge \e \big\} 
&\le \exp \big( r \ln (2ep/r) \big) \exp \big( - 5 u^2 r \ln(2ep) \big) \\
&\le \exp \big( - 4 u^2 r \ln(2ep) \big).
\end{align*}
Finally, we take the union bound over all allowed pairs $r,s$. Since $r \ge 1$ and $u \ge 1$, we have
\begin{align*}
\P \big\{ 
  \max_{\substack{r,s \in [p] \\ r \le s}} 
  \max_{\substack{x \in \Reg_p(r) \\ y \in \Reg_p(s)}}  
  \< (M \cdot \S'_n)x, y\> \ge \e \big\} 
&\le p^2 \exp \big( - 4 u^2 \ln(2ep) \big) \\
&\le \exp \big( - 2u^2 \ln(2ep) \big).
\end{align*}

Using \eqref{max max}, we have shown that
$$
\P \big\{ \|M \cdot \S'_n\| \ge 12 \lceil \ln(2p) \rceil^2 \e \big\} 
\le 2 \exp \big( - 2u^2 \ln(2ep) \big)
$$
for all $u \ge 1$ and for $\e = \e(u)$ defined in \eqref{epsilon}. 
Integration yields
\begin{equation}								\label{refined}
\E \|M \cdot \S'_n\| 
\le 84 \frac{\|M\|_{1,2}}{\sqrt{n}} \lceil \ln(2ep) \rceil^{5/2} 
  + 263 \frac{\|M\|}{n} \lceil \ln(2ep) \rceil^3.
\end{equation}
Decoupling Proposition~\ref{decoupling} completes the proof of Theorem~\ref{main}, giving also
an explicit bound on the absolute constant $C$ and a slightly better dependence on $p$.
\qed

\begin{remark}[Arbitrary distributions]
It is likely that the results of this paper generalize from Gaussian to arbitrary distributions 
in $\R^p$ with enough moments.
However, it is not clear whether a version of Decoupling Proposition~\ref{decoupling} holds
for general distributions. 
\end{remark}

\bibliographystyle{plain}
\bibliography{allref}

\end{document}